\documentclass[11pt]{amsart}

\usepackage{color,amssymb}
\usepackage{enumerate}
\usepackage{graphicx}

\input xy
\xyoption{all}

\newcommand{\F}{\mathcal F}
\renewcommand{\P}{\mathcal P}

\newcommand{\G}{\mathcal G}
\newcommand{\U}{\mathcal U}
\newcommand{\V}{\mathcal V}

\newcommand{\N}{\mathbb N}
\newcommand{\R}{\mathbb R}
\newcommand{\C}{\mathcal C}
\newcommand{\K}{\mathcal K}

\newcommand{\id}{\mathrm{id}}

\newcommand{\e}{\varepsilon}
\renewcommand{\phi}{\varphi}

\newcommand{\Lip}{\mathrm{Lip}}
\newcommand{\im}{\mathrm{im}}
\newcommand{\diam}{\mathrm{diam}}

\newtheorem{theorem}{Theorem}[section]
\newtheorem*{theorem*}{Main theorem} 

\newtheorem{lemma}[theorem]{Lemma}

\newtheorem{problem}[theorem]{Problem}
\newtheorem{corollary}[theorem]{Corollary}

\theoremstyle{definition}
\newtheorem{definition}[theorem]{Definition}
\newtheorem{remark}[theorem]{Remark}
\newtheorem{example}[theorem]{Example}

\title{Peano continua with self regenerating fractals}

\author{Magdalena Nowak}
\address{M.Nowak: Mathematics Department, Jan Kochanowski University in Kielce, ul. Uniwersytecka 7, 25-406 Kielce, Poland}
\email{magdalena.nowak805@gmail.com}

\subjclass[2020]{Primary: 28A80; Secondary: 37C25, 37C70, 47H09, 51F99, 54C05, 54D05, 54D30, 54F15}
\keywords{deterministic fractal, topological fractal, self regenerating fractal, topologically contracting system, iterated function system, IFS-attractor, Peano continuum}

\begin{document}
\begin{abstract}
We deal with the question of Masayoshi Hata: is every Peano continuum a topological fractal? A compact space $X$ is a topological fractal if there exists $\F$ a finite family of self-maps on $X$ such that  $X=\bigcup_{f\in\F}f(X)$ and for every open cover $\mathcal{U}$ of $X$ there is $n\in\mathbb{N}$ such that for all maps $f_1,\dots,f_n\in\F$ the set $f_1\circ\dots\circ f_n(X)$ is contained in some set $U\in\mathcal{U}$.

In the paper we present some idea how to extend a topological fractal and we show that a Peano continuum is a topological fractal if it contains so-called self regenerating fractal with nonempty interior. 
A Hausdorff topological space $A$ is a self regenerating fractal if for every non-empty open subset $U$, $A$ is a topological fractal for some family of maps constant on $A\setminus U$. 

The notion of self regenerating fractal much better reflects the intuitive perception of self-similarity. We present some classical fractals which are self regenerating.
\end{abstract}
\maketitle

\section{Introduction}\label{intro}

We deal with the topological generalization of IFS-attractors - a compact set invariant under the finite family of contractions. Let us recall that a function $f$ between two metric spaces is called \textit{contraction} if its Lipschitz constant $\Lip f<1$. By an \textit{iterated function system} (IFS) on a metric space $X$ we understand a finite family of contractions $X\to X$. A function $X\to X$ is called a~\textit{self-map} on~$X$. For a given family $\F$ of self-maps on $X$ we define the following families of maps:  $$\F^0=\{id_X\},\quad \F^n=\{f_1\circ\dots\circ f_n; f_1,\dots,f_n\in\F\}.$$ Moreover, let $\F(Y)=\bigcup_{f\in\F}f(Y)$ for $Y\subset X$.

Let $\F$ be an iterated function system on a metric space $X$. 
 A compact set $A\subset X$ is an \textit{IFS-attractor} (\textit{deterministic fractal, self-similar set}) for $\F$ if $A=\F(A)$. A simple example of such space is the unit interval $[0,1]$ -  the IFS-attractor for family $\{\frac{x}2,\frac{x+1}2 \}$. Other known examples are the ternary Cantor set, Koch curve, Sierpi\'nski triangle, Sierpi\'nski carpet, etc. It can be also shown that every arc of finite length is an IFS-attractor (see \cite{Sanders}).  

A \textit{topological fractal} is a topological version of IFS-attractor. It is a pair $(X,\F)$ where $X$ is a compact space, $\F$ is a finite family of continuous self-maps which has some topological contractive property on $X$ (see Definition \ref{def_top_contr}) and  $X=\F(X)$. Topological fractals have been studied in various contexts, eg. among countable spaces \cite{N} or zero-dimensional spaces \cite{BNS}. Now we are interested in topological fractals in the class of Peano continua. By Peano continuum we understand a metrizable, locally connected continuum or, equivalently (thanks to the Hahn–Mazurkiewicz theorem), a continuous image of the unit interval. Masayoshi Hata proved in \cite{H} that for every topological fractal $(X,\F)$ if $X$ is connected, then it is locally connected, so it is a Peano continuum. It is still an open question: is every Peano continuum a topological fractal?

Looking for the conditions when a Peano continuum $P$ becomes a topological fractal, we discover that the existence of a free arc (an open subset of $P$ homeomorphic to the interval) implies this result (see \cite{FN}). This leads us to the notion of self regenerating fractal. A Hausdorff topological space $X$ is called a \textit{self regenerating fractal} if for every nonempty, open subset $U$ there exists $\F$, a family of continuous functions constant outside $U$ such that $(X,\F)$ is a topological fractal. Having such a self regenerating fractal as a subset with nonempty interior, guarantees that a Peano continuum is a topological fractal together with some family of maps. This is our main result presented in the chapter \ref{main_chapter}:
\begin{theorem*}
	For every Peano continuum $X$ which has $A\subset X$ self regenerating fractal with nonempty interior, $X$ is an underlying space for some topological fractal.
\end{theorem*}
This statement and definition of self regenerating fractal are proposed by Taras Banakh and developed by the author.

\section{Topologically contracting systems}\label{top_contr_chapter}

The generalization of iterated function systems for topological spaces was proposed by Banakh and Nowak in \cite{BN1} as topologically contracting families of maps.

\begin{definition}\label{def_top_contr}
	A finite family $\F$ of continuous self-maps on the Hausdorff space $X$ is called a {\em topologically contracting system} on $Y\subset X$ if for any open cover $\U$ of $Y$ there exists a number $n$ such that for every $f_1,\dots,f_n\in\F$ the image $f_1\circ\dots\circ f_n(Y)$ is contained in some set $U\in\U$.
\end{definition}

\begin{remark}\label{top_contr_def_rem}
	If the set $Y\supset \F(X)$, then the definition of topologically contracting system $\F$ on $Y$ is equivalent to the condition that for any open cover $\U$ of $Y$ there exists a number $n$ such that for every $m\geq n$ and $f_1,\dots,f_m\in\F$ the image $f_1\circ\dots\circ f_m(Y)$ lies in some set from the cover $\U$.
\end{remark}


It can be easily shown that
\begin{remark}\label{rem_ifs}
	If $\F$ is a finite family of continuous self-maps on the metric space $X$ and $Y\subset X$ is compact such that
	\begin{itemize}
		\item $\F(X)\subset Y$
		\item maps from $\F$ are contractions on $Y$ 
	\end{itemize}
	then $\F$ is a topologically contracting system on $Y$.
\end{remark}

\begin{proof}
	Let us recall that for every open cover $\U$ of the compact metric space $Y$ there exists so-called Lebesgue number $\lambda$, which guarantees that every subset of $Y$ with diameter $<\lambda$ lies in some element of $\U$. Now take $\alpha = \max_{f\in\F}{Lip f|_Y}<1$ and the natural number $n$ such that $\alpha^n \cdot \diam(Y) < \lambda$. Then for every $f_1,\dots,f_n\in\F$, the image $f_1\circ\dots\circ f_n(Y)$ lies in some set from $\U$ because its diameter is $<\lambda$.
\end{proof}

Thus every iterated function system on the compact metric space $X$ is topologically contracting on $X$. Moreover, we can prove that every topologically contracting system on the Hausdorff space $X$ is also topologically contracting on $\F(X)$, if $\F(X)$ is closed. The area where a given family $\F$ is topologically contracting is in fact arbitrary if only it is a closed set lying between $\F(X)$ and $X$. It is shown in the following
\begin{lemma}\label{lem_top_contr}
	For a finite family $\F$ of continuous self-maps on the Hausdorff space $X$ and every nonempty, closed $Y$ such that $\F(X)\subset Y\subset X$ the following facts are equivalent
	\begin{enumerate}[(i)]
		\item $\F$ is topologically contracting system on $X$
		\item $\F$ is topologically contracting system on $Y$
	\end{enumerate} 
\end{lemma}

\begin{proof} \textbf{Step 1.} (i)$\Rightarrow$(ii)
	\newline Take $\U$ the open cover of $Y$. Then $\V=\U\cup\{X\setminus Y\}$ is an open cover of $X$, such that for every $V\in\V$ if $V\cap Y\neq\emptyset$ then $V\in\U$. From the assumption (i) there exists a natural number $n$, such that for every $f_1,\dots,f_n\in\F$ there exists $V\in\V$ such that $f_1\circ\dots\circ f_n(X)\subset V$. Due to the Remark \ref{top_contr_def_rem} we can assume that $n\geq 1$ so the image $f_1\circ\dots\circ f_n(X)$ is also a subset of $\F(X)\subset Y$. Moreover a nonempty set $f_1\circ\dots\circ f_n(Y)\subset f_1\circ\dots\circ f_n(X)\subset V\cap Y$ which means that $V\in\U$ and $\F$ is a topologically contracting system on $Y$.
	\newline \textbf{Step 2.} (i)$\Leftarrow$(ii)
	\newline Take $\U$ the open cover of $X$. It is also an open cover of $Y$ so we can find a number $n$, such that for every $f_1,\dots,f_n\in\F$ there exists $U\in\U$ such that $f_1\circ\dots\circ f_n(Y)\subset U$. Then for arbitrary $f_{n+1}\in\F$ the image $f_1\circ\dots\circ f_{n+1}(X)\subset f_1\circ\dots\circ f_n(\F(X))\subset f_1\circ\dots\circ f_n(Y)\subset U$. Thus $\F$ is a topologically contracting system on $X$.
\end{proof}

	Note that in the proof (i)$\Leftarrow$(ii) we do not use the closeness of the set $Y$ but only the fact that $\F(X)\subset Y\subset X$ so finally we have 
\begin{corollary}
	If $\F$ is topologically contracting on $Y$ and $\F(X)\subset Y\subset X$, then for every closed $Z$ and arbitrary $Z'$ where $\F(X)\subset Z\subset Y\subset Z'\subset X$, the system $\F$ is topologically contracting on $Z$ and on $Z'$.
\end{corollary}

\section{Topological fractals and their extensions} 

Assume that every compact set is also a Hausdorff space. 

\begin{definition}\label{def_top_fr}
	A pair $(X,\F)$ is called a {\em topological fractal} if $X$ is a compact space, $\F$ is a topologically contracting system on $X$ and $X=\bigcup_{f\in\F}f(X)$. Then the space $X$ and the family $\F$ will be called respectively an {\em underlying space} and a {\em fractal structure} of a topological fractal $(X,\F)$.
\end{definition}
Note that in case of topological fractal $(X,\F)$ we have $\F(X)=X$, so results from the chapter \ref{top_contr_chapter} are trivial. Nevertheless, they will be useful if we want to "extend a topological fractal" to the bigger set. Consider the following problem:
\begin{problem}
	Let $(Y,\F)$ be a topological fractal and $Y\subset X$. How to construct a topologically contracting system $\G$ on $X$ (contains extensions of maps from $\F$) such that $(X,\G)$ is a topological fractal?
\end{problem}
Now we have to extend maps from $\F$ to the set $X$ and maybe define an another family $\P$ of self-maps on $X$. Then the situation that $\F(X)\neq X$ and $\P(X)\neq X$ is possible. Let $\im p = p(X)$ be an image of the map $p\colon X\to X$. We obtain the following result:
\begin{theorem}\label{thm_!}
	Let $X$ be a compact metric space and $\F\cup\P$ be a finite family of continuous maps on $X$ s.t $X=\bigcup_{f\in\F\cup\P}f(X)$ then the following implications hold:
	\begin{equation}
	\tag{!}\label{!!}
	\begin{cases}
	(i) & \F \text{ is a topologically contracting system on }X\\
	(ii)& \forall f\in \F\cup\P \quad \forall p\in\P \quad f(\im p) \text{ is a singleton }
	\end{cases}	 
	\end{equation}
	$$\Downarrow$$
	\begin{equation}
	\tag{!!}\label{!!!}
	\begin{cases}
	(i) & \F \text{ and } \P \text{ are topologically contracting systems on }X\\
	(ii)& \forall f\in \F \quad \forall p\in\P \quad f(\im p) \text{ is a singleton}
	\end{cases}	 
	\end{equation}
	$$\Downarrow$$
	$$(X,\F\cup\P) \text{ is a topological fractal.}$$
	
\end{theorem}
\begin{proof}
	\textbf{Step 1.} (\ref{!!})$\Rightarrow$(\ref{!!!})\newline
	Due to the assumption (ii) we have that for all $p,q\in\P$ and $f\in\F$ sets $f(\im p)$ and $q(\im p)$ are singletons. Thus the singleton $(q\circ p)(X)$ lies in some set from the arbitrary cover of $X$, so $\P$ is topologically contracting on $X$.
	
	\textbf{Step 2.} (\ref{!!!})$\Rightarrow (X,\F\cup\P) \text{ is a topological fractal}$\newline
	For $\U$, an open cover of $X$, there exist five constants: 
	\begin{enumerate}[(1)]
		\item $\exists~n_1\in\N$ such that $\forall f_1,\dots,f_{n_1}\in \F$ the image $f_1\circ\dots\circ f_{n_1}(X)$ lies in some set from $\U$.
		\item $\exists~n_2\in\N$ such that $\forall p_1,\dots,p_{n_2}\in \P$ the image $p_1\circ\dots\circ p_{n_2}(X)$ lies in some set from $\U$.
		\item $\exists~\e>0$, the Lebesgue number of $\U$ (so each set with diameter $<\e$ lies in some set from $\U$). 
		\item consider $\P^{<n_2}=\bigcup_{n=0}^{n_2-1}\P^n$, the family  of continuous functions on the compact space $X$. All such maps are uniformly continuous. Thus there exists $\delta>0$ such that for all $f\in\P^{<n_2}$ and every $Y\subset X$ 
		$$\diam(Y)<\delta \Rightarrow \diam f(Y)<\e$$
		so from (3) $f(Y)$ lies in some set from $\U$. 
		\item Take $\V=\{B(x,\frac{\delta}{2})\}_{x\in X}$  the cover of $X$ containing open balls of diameter $<\delta$. Then there exists $n_3\in\N$ such that for all $f_1,\dots,f_{n_3}\in \F$ the image $f_1\circ\dots\circ f_{n_3}(X)$ lies in some set from $\V$ so its diameter is $<\delta$.
	\end{enumerate}
	Then for $n = \max\{n_1, 2n_2, 2n_3\}$ and every $h_1,\dots,h_{n}\in \F\cup\P$ we consider the following cases: 
	\begin{enumerate}[1.]
		\item if $h_1,\dots,h_{n}\in \F$ then $h_1\circ\dots\circ h_{n}(X)$ lies in some set from $\U$ because $n\geq n_1$ and (1);
		\item if there exists an index $i\in\{1,\dots,n-1\}$ such that $h_i\in\F$ and $h_{i+1}\in\P$, then $h_1\circ\dots\circ h_{n}(X)\subset h_1\circ\dots\circ h_i(\im h_{i+1})$ which is a singleton from assumption (ii), so it lies in some set from $\U$; 
		\item if there exists $k\leq n$ such that $h_1,\dots,h_k\in\P$ and $h_{k+1},\dots,h_n\in\F$, then
		\begin{enumerate}[a)]
			\item for $k\geq n_2$ the image $h_1\circ\dots\circ h_{n}(X)\subset h_1\circ\dots\circ h_{k}(X)$ and from (2) both lie in some set from $\U$; 
			\item for $k< n_2$ we have $n-k\geq n_2+n_3-k >n_3$ so by (5) $\diam h_{k+1}\circ\dots\circ h_n(X)<\delta$ and $h_{1}\circ\dots\circ h_k\in\P^{<n_2}$. Then from (4) a diameter of the image $h_{1}\circ\dots\circ h_k(h_{k+1}\circ\dots\circ h_n(X))$ is $<\e$ so it lies in some set from $\U$.
		\end{enumerate}
	\end{enumerate}
\end{proof}
Hence we have
\begin{corollary}\label{col_!}
	Let $(Y,\F)$ be a topological fractal. A pair $(Y\cup Z,\F'\cup\P)$ is a topological fractal if 
	\begin{enumerate}[1)]
		\item $Y\cup Z= \bigcup_{f\in \F'\cup\P}f(Y\cup Z)$
		\item $\F'$ is a topologically contracting system on $Y\cup Z$ 
		\item $\forall f\in \F'\cup\P \quad \forall p\in\P \quad f(\im p)$ is a singleton.
	\end{enumerate} 
\end{corollary}	
The family $\F'$ usually contains extensions of maps from $\F$, such that $Y=\F(Y) =\F'(Y\cup Z)$ and $Z=\P(Y)=\P(Y\cup Z)$. Then the set $Z$ is a finite union of continuous images of $Y$, for example: 
\begin{itemize}
	\item $Y$ is an arc and $Z$ is a finite union of Peano continua;
	\item $Y$ is a convergent sequence and $Z$ is a scattered space of finite set of accumulation points;
	\item $Y$ is the Cantor set and $Z$ is an arbitrary compact space.
\end{itemize} 
The union of these sets satisfies the assumption of Corollary \ref{col_!} if they are disjoint or have only one common element. See figure \ref{fig_empty} and the following statement: 

\begin{lemma}\label{ex1}
	For every topological fractal $(Y,\F)$ and $Z$ - finite union of continuous images of $Y$ such that $Y\cap Z$ is an empty set or a singleton, there exists a family $\F'\cup\P'$ such that $(Y\cup Z,\F'\cup\P')$ is a topological fractal. 
\end{lemma}

\begin{proof}
Due to the fact that $Z$ is a finite union of continuous images of $Y$, denote by $\P$ the family of continuous maps $Y\to Z$ such that $Z=\P(Y)$. Now we define families $\F'$ and $\P'$ as follows.\\ \\
If $Y\cap Z = \emptyset$, take an arbitrary $y_0\in Y$ and $z_0\in Z$. Then for every $f\in \F$ and $p\in \P$ let $f',p'\colon Y\cup Z\to Y\cup Z$ be the maps such that 
$$f'(x)=
\begin{cases}
f(x); & x\in Y \\
y_0;&  x\in Z
\end{cases}
\quad \quad
p'(x)=
\begin{cases}
p(x); & x\in Y \\
z_0;&  x\in Z.
\end{cases}
$$
If $Y\cap Z = \{x_0\}$ then 
$$f'(x)=
\begin{cases}
f(x); & x\in Y \\
f(x_0);&  x\in Z
\end{cases}
\quad \quad
p'(x)=
\begin{cases}
p(x); & x\in Y \\
p(x_0)&  x\in Z.
\end{cases}
$$
Then $\F'=\{f'; f\in \F\}$ and $\P'=\{p' ; p\in\P\}$ satisfy all assumptions of Corollary \ref{col_!} which ends the proof. 
\end{proof}

Figure \ref{fig_empty} shows the idea used in the proof and some examples of underlying spaces for some topological fractal.
\begin{figure}[h]
	\includegraphics[scale=0.4]{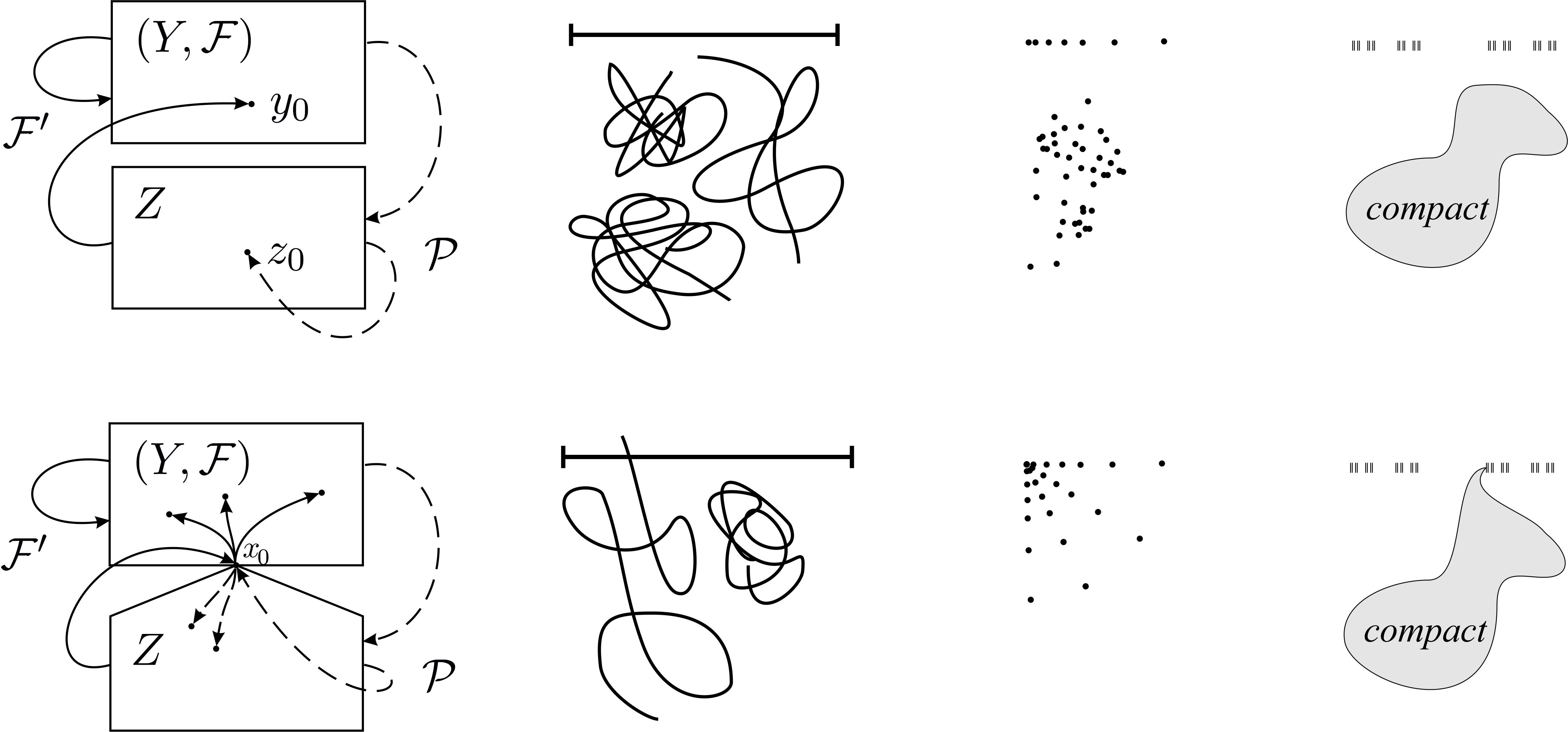}
	\caption{An outline of the proof of the Lemma \ref{ex1} and examples for $Y\cap Z = \emptyset$ and $Y\cap Z = \{x_0\}$}\label{fig_empty}
\end{figure}

When the intersection of sets $Y$ and $Z$ has more than one point the situation is more complicated. 
\begin{remark}
	There exist examples of topological fractal $(Y,\F)$ and a set $Z$ where $Y\cap Z$ is finite and $(Y\cup Z,\F'\cup\P')$ is a topological fractal for some fractal structure $\F'\cup\P'$.
\end{remark}
For example $Y\cup Z$ is a Peano continuum and $Y$ is the closure of so-called free arc. This example had been described in \cite{FN}. The idea of the proof is shown in the Figure \ref{free_arc}.
\begin{figure}[h]
	\includegraphics[scale=0.4]{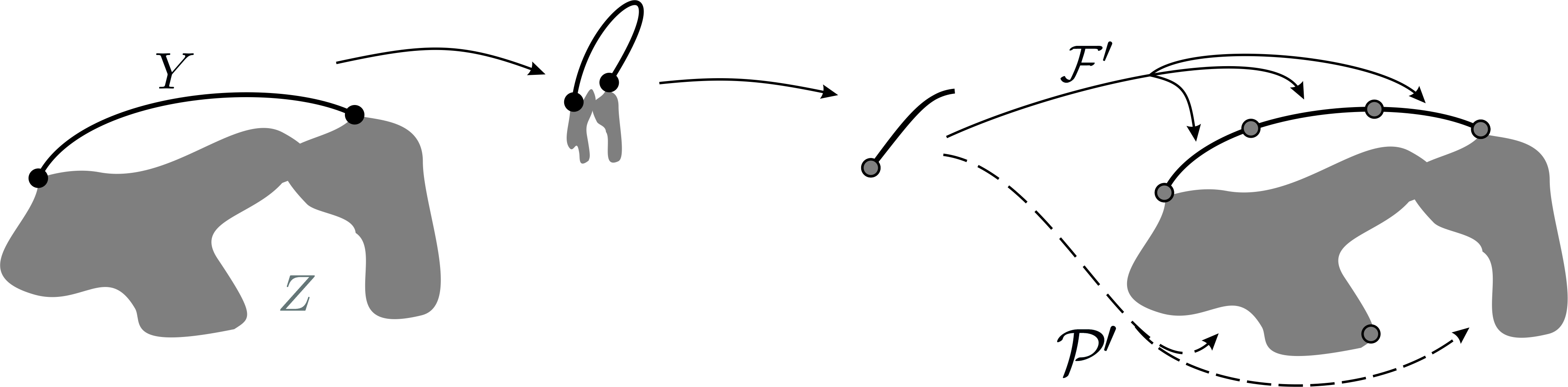}
	\caption{A Peano continuum with a free arc is a topological fractal }\label{free_arc}
\end{figure}

The assumption (\ref{!!!}) from Theorem \ref{thm_!} is more symmetric than (\ref{!!}). Using it we obtain also 
\begin{corollary}\label{col_!!}
	Let $(Y,\F)$ and $(Z,\P)$ be topological fractals and $Y,Z$ be subsets of the same topological space. A pair $(Y\cup Z,\F'\cup\P')$ is a topological fractal if 
	\begin{enumerate}[1)]
		\item $Y\cup Z= \bigcup_{f\in \F'\cup\P'}f(Y\cup Z)$
		\item $\F'$ and $\P'$ are topologically contracting systems on $Y\cup Z$ 
		\item $f(\im p)$ is a singleton for all $f\in\F', p\in\P'$.
	\end{enumerate} 
\end{corollary}	
Again, the families $\F'$ and $\P'$ usually contain extensions of maps from $\F$ and $\P$, and $Y=\F(Y) =\F'(Y\cup Z)$ and $Z=\P(Z)=\P'(Y\cup Z)$.

Looking for sufficient conditions for being underlying space for some topological fractal, and using Lemma \ref{ex1}, we obtain the following lemma which can be easily proved inductively.

\begin{lemma}
	A finite union of $\{X_i\}_{i=1}^n$ underlying spaces of topological fractals such that $X_i\cap X_j$ is an empty set or a singleton for every $i\neq j$, is underlying space for some topological fractal. 
\end{lemma}

\begin{remark}
	Note that the assumption 2) in both Corollary \ref{col_!} and Corollary \ref{col_!!} is satisfied if $\F'$ is topologically contracting system on $Y=\F'(Y\cup Z)$ (and $\P'$ - topologically contracting on $Z=\P'(Y\cup Z)$). Then due to the Lemma \ref{lem_top_contr} we obtain its (their) topological contractivity on the whole union $Y\cup Z$.  
\end{remark}

\section{Peano continua as topological fractals}\label{main_chapter} 

We are interested in the old question posted by Hata in 1985:
\begin{problem}\label{hata}
	Is every Peano continuum an underlying space for some topological fractal?
\end{problem}

 In fact Hata in his paper \cite[page 392]{H} asked is every locally connected continuum $Q$ is an invariant set ($\F(Q)=Q$) for the finite family $\F$ of so-called weak contractions. Now we know  that such space $Q$ with family $\F$ is exactly the same as topological fractal. Namely, a space $Q$ is an underlying space for some topological fractal if and only if there exists a finite family $\F$ of weak contractions such that $\F(Q)=Q$ (see \cite[6.4]{BKNNS}).

In the \cite{FN} we show that a Peano continuum $X$ with a free arc is an underlying space for some topological fractal. The idea of the proof, shown in Figure \ref{free_arc}, is to use a fractal structure $\F$ on the arc $Y$ and extend it to the whole $X$, like in the Corollary \ref{col_!}. Moreover, using continuous images of arc $Y$ we could cover the whole $X\setminus Y$. 

This prompted us to look for special subsets of a Peano continuum which are underlying space for some topological fractal and their continuous images can cover the whole Peano continuum. We also want to use continuous maps which are constant outside some open set such that we can guarantee that the assumption (ii) from Theorem \ref{thm_!} will be met. It is easy to prove that
\begin{remark}\label{ext}
	For $U\subsetneq A\subset X$ where $U$ is open, $A$ is a closed subset of topological space $X$ and for every continuous map $f\colon A\to A$ constant on $A\setminus U$ (so $f(A\setminus U)=\{x_0\}$), we can construct its continuous extension $\bar{f}\colon X\to X$ constant on $X\setminus U$ in the following way: for every $x\in X$
	$$\bar{f}(x)=
	\begin{cases}
	f(x) &, x\in U\\
	x_0 &, x\in X\setminus U.
	\end{cases}
	$$
\end{remark}
Note that the construction of $\bar{f}$ implies that $\bar{f}(X)=A$.

Here we present one of our main theorems for a Peano continuum, which gives sufficient conditions for being the underlying space for some topological fractal. 	
\begin{theorem}\label{thm_top_fr}
	A Peano continuum $X$ is an underlying space for some topological fractal if there exists a topological fractal $(A,\F)$ such that $A\subset X$ and if there exists a nonempty open set $U\subsetneq A$ such that the following assumptions hold
		\begin{enumerate}[(i)]
			\item\label{t1} all maps from $\F$ are constant on $A\setminus U$
			\item\label{t3} there exists $\bigcup_{i=1}^n P_i$ finite union of Peano continua such that $X\setminus A\subset \bigcup_{i=1}^n P_i \subset X\setminus U$.
		\end{enumerate}
\end{theorem}

\begin{proof} 
	Due to the Remark \ref{ext} we can assume that $\F$ contains continuous self-maps on $X$, constant outside the set $U$. Then $\F$ is a topologically contracting system on $X$ and $A=\F(X)$. We shall construct a finite family $\P$ of continuous self-maps on $X$ such that:
	\begin{enumerate}[a)]
		\item $\P(X) =\bigcup_{i=1}^n P_i $
		\item for all $p\in\P$ the map $p|_{X\setminus U}=const$.
	\end{enumerate} 
	There is nothing to prove if $X$ is a singleton. Then $(X,\id_X)$  is a topological fractal. 
	
	If $|X|>2$, then every open subset $U$ of Peano continuum $X$ is infinite. Thus the finite union $\bigcup_{f\in\F}f(X)=\F(X)=A\supset U$ is also infinite. This means that there exists $f_0\in\F$ such that $f_0(X)$ has at least two points. It is also a compact and connected set as a continuous image of the Peano continuum $X$. 
	Now we take a continuous map $\phi\colon f_0(X) \to \R$ such that its image $\im \phi$ is a continuum with at least two points (its existence is a result of Tietze extension theorem). This means that $\im \phi$ is an interval (we can assume that $[0,1]$) - see Figure \ref{proof}.
	\begin{figure}[h]
		\includegraphics[scale=0.8]{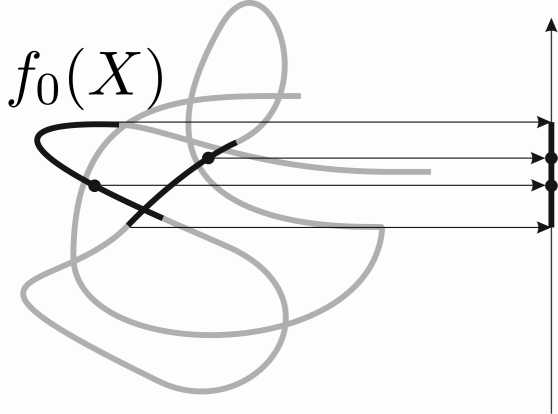}
		\caption{$\phi\colon f_0(X) \to \R$}\label{proof}
	\end{figure}
	
	For every $i=1,\dots,n$ and Peano continuum $P_i$ there exists an embedding $e_i\colon [0,1]\to P_i$. Now define $\P=\{e_i\circ\phi\circ f_0\}_{i=1}^n$, a finite family satisfying a) and b). This implies that $X=\F(X)\cup\P(X)$ and $f(\im p)$ is a singleton for all $f\in\F\cup\P$ and $p\in\P$. Family $\F\cup\P$ satisfies the assumption (\ref{!!}) from Theorem \ref{thm_!} so $(X,\F\cup\P)$ is a topological fractal.
	
\end{proof}
Looking for the topological fractal $(A,\F)$ satisfying assumptions from the above theorem we have to remember about the following
\begin{itemize}
	\item $A$ is a finite union of compact and connected sets (because $A=\F(X)$);
	\item if $X$ is infinite then also $A$ is infinite (because it contains an open set).
\end{itemize}

Several observations of a Peano continuum will lead us to the conclusion that the condition (ii) from Theorem \ref{thm_top_fr} is always true for some $U$ open subset of a given set $A$ with nonempty interior. First, let us recall that a Peano continuum is a continuous image of the unit interval and every continuous map on the compact space is uniformly continuous. Such uniform continuity gives us the following 
\begin{lemma}\label{engelking}
	Every Peano continuum in every metric is a finite union of Peano continua of an arbitrary small diameter.
\end{lemma}

Now we prove the announced proposition
\begin{lemma}\label{lem_finite_union}
	Let $X$ be a Peano continuum and $A\subset X$ has nonempty interior. Then there exists a nonempty open set $U\subset A$ such that $X\setminus U$ is a finite union of Peano continua.
\end{lemma}

\begin{proof}
	 Take an arbitrary metric $d$ on $X$. The open set $int A\neq\emptyset$ so there exists an element $x_0\in X$ and a positive radius $\e>0$ such that the open ball $B(x_0, \e)\subset int A$. Due to the Lemma \ref{engelking} there exist Peano continua $P_1, P_2,\dots,P_n$ of diameters $<\e$ such that $X=\bigcup_{i=1}^n P_i$. We can assume that $P_i \not\subset\bigcup_{j\neq i}P_j$ for all $i=1,\dots,n$. Thus we have for every $i=1,\dots,n$
	 \begin{enumerate}[(a)]
	 	\item $0<\diam P_i<\e$
	 	\item $P_i \setminus \bigcup_{j\neq i}P_j \neq\emptyset$.
	 \end{enumerate} 
	 Now, take an index $i_0\in\{1,\dots,n\}$ such that $x_0\in P_{i_0}$. By (a) we obtain that $P_{i_0}\subset B(x_0,\e)\subset int A$ so $X\setminus A\subset \bigcup_{j\neq i_0}P_j$. Define an open set $U:=P_{i_0}\setminus \bigcup_{j\neq i_0}P_j = X\setminus\bigcup_{j\neq i_0}P_j$. By (b) it is a nonempty subset of $A$ 
	 and $X\setminus U=\bigcup_{j\neq i_0}P_j$ is a finite union of Peano continua. 
	 
	 \end{proof}

\begin{remark}
	It is not true that for an arbitrary open set $U\subset X$, the thesis of Lemma \ref{lem_finite_union} holds. The counterexample may be the following: 
	$X=[0,1]$ and take the strictly decreasing sequence $(a_n)\subset X$. Then the set $U=\bigcup_{n=0}^\infty (a_{2n},a_{2n+1})$ is open but $X\setminus U$ is the infinite union of disjoint intervals. 
\end{remark}

	This is the reason why we need to consider the existence of the set $A$ where, for an arbitrary open subset $U$, we can a find fractal structure $\F$ satisfying the condition (i) from Theorem \ref{thm_top_fr}. This leads us to the notion of self regenerating fractal.
	
\begin{definition}\label{def_top_super}
	A Hausdorff topological space $A$ is called a {\em self regenerating fractal} if for every nonempty, open set $U\subset A$ there exists a family $\F$ of continuous self-maps on $A$, constant outside $U$, such that $(A,\F)$ is a topological fractal.
\end{definition}

	Existence of such space inside a Peano continuum $X$ guarantees that $X$ is an underlying space for some topological fractal. Indeed, we have the following main theorem 
	 
\begin{theorem}\label{main_thm}
	For every Peano continuum $X$ which has $A\subset X$ self regenerating fractal with nonempty interior, $X$ is an underlying space for some topological fractal.
\end{theorem}

\begin{proof}
	Due to the Lemma \ref{lem_finite_union} we can find an open nonempty set $U\subset A$ such that $X\setminus U$ is a finite union of Peano continua, so it satisfies the assumption (\ref{t3}) from Theorem \ref{thm_top_fr}. Then by the definition of self regenerating fractal there exists  a family $\F$ of continuous self-maps on $A$, constant outside $U$, such that $(A,\F)$ is a topological fractal. Thus we obtain the assumption (\ref{t1}) and the Theorem \ref{thm_top_fr} implies that $X$ is an underlying topological space for some topological fractal.
\end{proof}

Therefore we went a step further in finding the answer to the Hata's problem \ref{hata}, but we can ask another two questions. 

\begin{problem}
	Has every Peano continuum which is an underlying space for some topological fractal, a self regenerating fractal as a subset with nonempty interior? 
\end{problem} 
In other words: is the converse of Theorem \ref{main_thm} true? 

\begin{problem}
	Is there a nontrivial Peano continuum without a self regenerating fractal as a subset with nonempty interior?
\end{problem} 

\section{Self regenerating fractals}

In this chapter we will expand our knowledge about self regenerating fractals which plays an important role in the main theorem of this paper. Let us notice that the definition of self regenerating fractal tries to catch the sense of self-similarity - every small piece of the space is a "copy" of the whole object. The structure of a self regenerating fractal must be rich enough for its every piece to be able to self regenerate a finite cover of the whole object. Such property appears for example in classical self-similar fractals like the Cantor set, Sierpi\'nski triangle, etc. Each of their open subsets contains the small copy of the figure - the building block for the whole set. This leads us to the notion of (self-similar) brick.

\begin{definition}
	Let $X$ be a topological space and $B\subset X$ is closed. A pair $(B,\P)$ is called a {\em brick} of $X$ if $\P$ is a finite family of continuous maps on $B$ and $\P(B)\cup B= X$.
	\newline Moreover if $B$ is the attractor for IFS $\F$, then the triple $(B,\F,\P)$ is called {\em self-similar brick}.
\end{definition}

\begin{theorem}\label{thm_brick}
	Let $X$ be a compact metric space and $(B,\F,\P)$ is a self-similar brick. Let $C:=\P(B)$. Then the following implications hold:
		\begin{equation}
		\tag{***}\label{***}
		\exists \phi\colon X\to B \text{ Lipschitz on }	 B \text{, continuous surjection, constant on }C
		\end{equation}
		$$\Downarrow$$
		\begin{equation}
		\tag{**}\label{**}
		\forall f\in\F\quad \exists \phi_f\colon X\to f(B) \text{ Lipschitz on }	 B \text{, continuous surjection, constant on }C 
		\end{equation}
		$$\Downarrow$$
		\begin{equation}
		\tag{*}\label{*}
		\begin{cases}
		(1) & \exists \P' \text{ finite family of continuous maps }X\to C \text{, constant on }C \text{ s.t. }\P'(X)=C.\\
		(2) & \exists \F' \text{ finite family of continuous maps }X\to B \text{, Lipschitz on }B \text{ and constant on }C \text{ s.t.} \\
		 & (a) \quad\F'(X)=B\\
		 & (b) \quad\forall f\in\F'\quad \Lip f|_B<1
		\end{cases}	 
		\end{equation}
		$$\Downarrow$$
		$$(X,\F'\cup\P') \text{ is a topological fractal and }\F'\cup\P' \text{contains maps constant on }C$$
		
\end{theorem}

\begin{proof} \textbf{Step 1.} (\ref{***})$\Rightarrow$(\ref{**})\newline
	For every $f\in\F$ the map $\phi_f:=f\circ\phi$ is a continuous surjection $X\to f(B)$, Lipschitz on $B$ and constant on $C$. 
	
	\textbf{Step 2.} (\ref{**})$\Rightarrow$(\ref{*})\newline 
	Note that $B=\bigcup_{f\in\F}f(B)=\bigcup_{f\in\F}\phi_f(X)$. Take $\alpha=\max_{f\in\F}\Lip f<1$ and $\beta =\max_{f\in\F}\Lip (\phi_f|_B)$. There exists a number $k\in\N$ such that $\alpha^k\beta<1$. Define $\F'=\{g\circ\phi_f; g\in\F^k, f\in\F\}$ - a finite family of maps $X\to B$ Lipschitz on $B$, constant on $C$ and satisfying 
	\begin{enumerate}[(a)]
		\item $\F'(X)=\bigcup_{h\in\F'}h(X)= \bigcup_{g\in\F^k}\bigcup_{f\in\F}(g\circ\phi_f)(X)=\bigcup_{g\in\F^k}g(\bigcup_{f\in\F}\phi_f(X))=\bigcup_{g\in\F^k}g(B)=B$;
		\item for every $h\in\F'$ its Lipschitz constant on $B$ is $\Lip (h|_B)=\Lip (g\circ\phi_f|_B)\leq \alpha^k\beta<1$.
	\end{enumerate}
	This gives us an assumption (2). Define also $\P'=\{p\circ\phi_f; p\in\P, f\in\F\}$ - a finite family of continuous maps $X\to C$, constant on $C$ such that $$\P'(X)=\bigcup_{h\in\P'}h(X)= \bigcup_{p\in\P}\bigcup_{f\in\F}(p\circ\phi_f)(X)=\bigcup_{p\in\P}p(\bigcup_{f\in\F}\phi_f(X))=\bigcup_{p\in\P}p(B)=\P(B)=C.$$ Thus we obtain (1).

	\textbf{Step 3.} (\ref{*})$\Rightarrow (X,\F'\cup\P')$ is a topological fractal where $\F'\cup\P'$ contains maps constant on $C$ \newline
	Note that $X=B\cup C= \F'(X)\cup\P'(X)=\bigcup_{f\in\F'\cup\P'}f(X)$ and $\F'$ (by Remark \ref{rem_ifs}) is topologically contracting on $B$ (so also on $X$). Moreover each map $f$ from $\F'\cup\P'$ is constant on $C$ so for every $p\in\P$ the set $f(\im p)$ is a singleton. This means that $\F'\cup\P'$ satisfies an assumption (\ref{!!}) from Theorem \ref{thm_!} which gives the thesis.
\end{proof}

Using this theorem and noting that $X\setminus U\subset X\setminus B\subset C$ for every $U\supset B$, we obtain the following 

\begin{corollary}\label{col_superfractal}
	A compact metric space $X$ is a self regenerating fractal if for every $U$ open set in $X$ there exist $B\subset U$ and families $\F,\P$ such that $(B,\F,\P)$ is a self-similar brick of $X$ satisfies (\ref{***}),(\ref{**}) or (\ref{*}).
\end{corollary}

\subsection{Examples of the self regenerating fractals} 

In this subsection we present several examples of self regenerating fractals and spaces that are not such fractals. Most of them satisfy the condition (\ref{***}) from Corollary \ref{col_superfractal}. Note that an open set using in all examples here is related to the subspace topology. 
\begin{example}
	The most trivial self regenerating fractal is \textbf{a singleton}. \\The only one fractal structure on it is the identity map. 
\end{example}

\begin{example}
	\textbf{The ternary Cantor set} $\C$ is a self regenerating fractal.\\  In every $U$ open subset of $\C$ we can find its small copy $B$ which creates a self-similar brick $(B,\F,\P)$ together with families of affine maps where $\P(B)=\C\setminus B$. Then it is easy to find a function $\phi$ satisfying the assumption (\ref{***}) - as in Figure \ref{cantor}. 
	\begin{figure}[h]
		\includegraphics[scale=0.5]{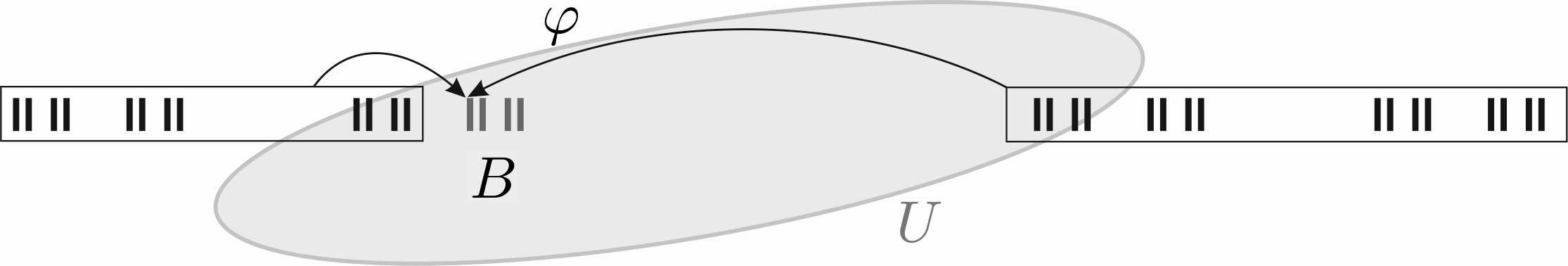}
		\caption{$\phi$ is the identity on $B$, continuous function, constant on $\C\setminus B$}\label{cantor}
	\end{figure}
	
	By Corollary \ref{col_superfractal} we obtain that $\C$ is a self regenerating fractal.
\end{example}

Below we present some IFS-attractors connected and locally connected which are self regenerating fractals. This can generate many examples of Peano continua that become topological fractals according to the Theorem \ref{main_thm}.

\begin{example}
	\textbf{The unit interval} $I=[0,1]$ is a self regenerating fractal. \\ Indeed, for every nonempty, open set $U\subset I$ there exists the interval $[a,b]\subset U$. Let families $$\F=\{f_1\colon[a,b]\to\Big[a,\frac{b+a}2\Big], f_2\colon[a,b]\to\Big[\frac{b+a}2,b\Big] \}\text{ and }\P=\{p_1\colon[a,b]\to[0,a], p_2\colon[a,b]\to[b,1]\}$$ contain affine maps from interval $[a,b]$ to another one. Then $([a,b],\F,\P)$ is self-similar brick where $\P([a,b])=I\setminus(a,b)$. Take  $\psi:I\to[a,\frac{b+a}2]$ such that 
	$$\psi(x)=
	\begin{cases}
	x & \text{for }x\in[a,\frac{b+a}2]\\
	a+b-x & \text{for }x\in(\frac{b+a}2,b] \\
	a & \text{for }x\notin[a,b].
	\end{cases}
	$$ 

	It is a continuous and non-expanding map constant outside $(a,b)$. Then the map $\phi\colon I\to [a,b]$ defined $\phi(x)= 2(\psi(x)-a)+a$ satisfies the assumption (\ref{***}). The graph of $\phi$ is the following:
	\begin{center}
		\begin{picture}(150,90)(-30,-10)
		\put(-20,0){\vector(1,0){130}}
		
		\put(0,-10){\vector(0,1){90}}
		\put(0,20){\line(1,0){20}}
		\put(20,20){\line(1,2){20}}
		\put(40,60){\line(1,-2){20}}
		\put(60,20){\line(1,0){40}}
		
		\put(20,-10){$a$}
		\put(20,-3){\line(0,1){6}}
		\put(60,-10){$b$}
		\put(60,-3){\line(0,1){6}}
		\put(100,-10){$1$}
		\put(100,-3){\line(0,1){6}}
		
		\put(-9,18){$a$}
		\put(-9,58){$b$}
		\put(-3,60){\line(1,0){6}}
		\put(57,37){$\phi$}
		\end{picture}
	\end{center}
	Thus, by Corollary \ref{col_superfractal}, we obtain that $I$ is a self regenerating fractal.
\end{example}

\begin{example}
	\textbf{The Koch curve} $\K$ is a self regenerating fractal. \\ 
	In every $U$ open set in $\K$ there exists $B$, a small copy of Koch curve. This leads us to the self-similar brick where $\F$ is a classical IFS for Koch curve scaled to $B$ and $\P$ is a family of similarities such that $\P(B)=\overline{\K\setminus B}$. Using a symmetry map and transforming the complement of $B$ into one point we obtain a continuous map non-expanding on $B$ and constant outside $B$. After scaling its image to the set $B$, we get the transformation $\phi$ from assumption (\ref{***}). See Figure \ref{koch}.
	\begin{figure}[h]
		\includegraphics[scale=0.35]{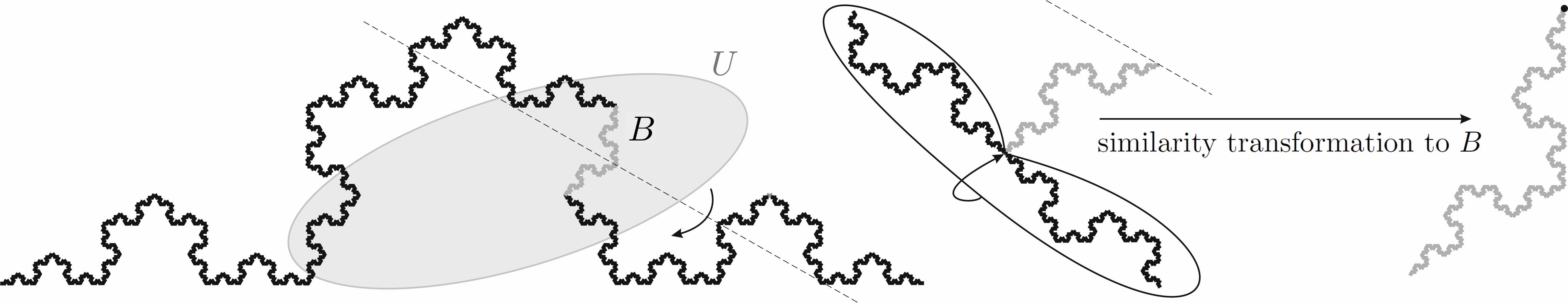}

		\caption{Construction of the continuous surjection $\phi\colon \K\to B$, Lipschitz on $B$ and~constant on $\overline{\K\setminus B}$. This  set (black) is transformed into one point.}\label{koch}
	\end{figure}
\end{example}

\begin{example}
	\textbf{The Sierpi\'nski triangle} $T$ is a self regenerating fractal. 
	\\ 
	In every $U$ open set in $T$ there exists $B$, a small copy of Sierpi\'nski triangle. This leads us to the self-similar brick $(B,\F,\P)$ where $\F$ is a standard (scaled) IFS for $T$ and $\P$ contains translations such that $\P(B)= \overline{T\setminus B}$. The continuous surjection $\phi\colon T\to B$ is a composition of two axial symmetries, retraction constant on $\overline{T\setminus B}$ and similarity transformation - like in Figure \ref{triangle}. 
	\begin{figure}[h]
		\includegraphics[scale=0.28]{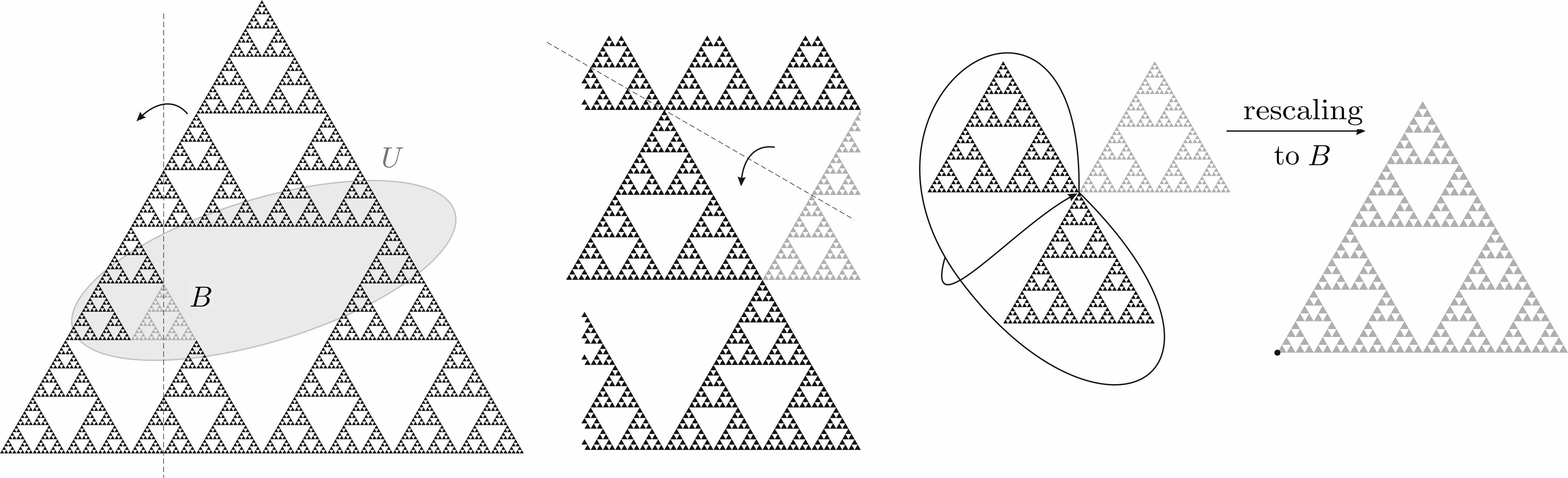}
		\caption{Construction of the continuous surjection $\phi\colon T\to B$ Lipschitz on $B$ and~constant on $\overline{T\setminus B}$ (black part of figure), satisfying the assumption (\ref{***}).}\label{triangle}
	\end{figure}
\end{example}

\begin{example}
	\textbf{The Sierpi\'nski carpet} $S$ is a self regenerating fractal.  \\ In its every open subset $U$ we can find $B$, a small affine copy of Sierpi\'nski carpet. This gives us the self-similar brick $(B,\F,\P)$ where $\P(B)= \overline{S\setminus B}$. Now we can construct the map from the assumption (\ref{***}): it is a composition of three axial symmetries, several metric projection and a similarity transformation (see Figure \ref{carpet}). All those maps are Lipschitz. By a metric projection on the plane into the convex set $Y\subset \R^2$ we understand a map $f\colon \R^2\to Y$ where  $d(x,f(x))=\inf\{d(x,y); y\in Y\}$ for every $x\in\R^2$. 
	\begin{figure}[h]
		\includegraphics[scale=0.27]{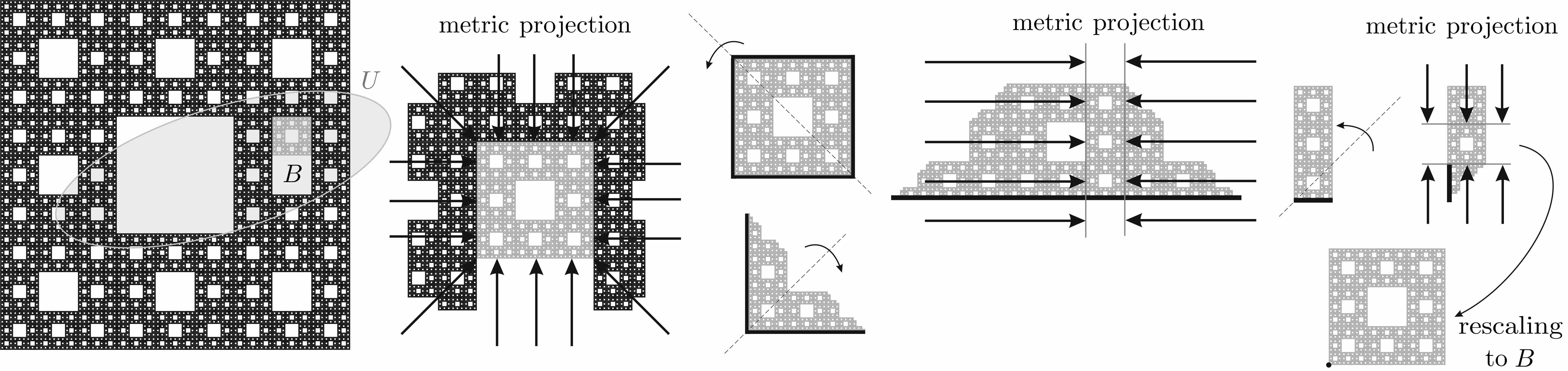}
		\caption{Construction of the continuous surjection $\phi\colon S\to B$ Lipschitz on $B$ and constant on $\overline{S\setminus B}$. This set (black part) is transformed into the one point.}\label{carpet}
	\end{figure}
\end{example}

The following interesting questions remain open:
\begin{problem}
	Are Menger cube and Barnsley fern self regenerating fractals?
\end{problem}

\begin{problem}
	Find an example of Peano continuum which is not a self regenerating fractal?
\end{problem}

We can present some non-connected examples of topological fractals that are not self regenerating fractals:
\begin{enumerate}
	\item $X\cup Y$ where $X,Y$ are disjoint sets with respectively finite and infinite cardinality, for example $[0,1]\cup\{2\}$ or  $\{\frac{1}{n}\}_{n\in\N}$,
	\item $X\cup Y$ where $X,Y$ are disjoint sets with respectively finite and infinite connected components, for example	$I\cup\C$ (interval and the Cantor set) or $\{\frac{1}{n}\}_{n\in\N}$.
\end{enumerate}

Moreover the Open Set Condition not guarantees that the space will be a self regenerating fractal. Let us recall that a topological fractal $(X,\F)$ satisfies OSC if there exists an open $V\subset X$ such that $\{f(V); f\in \F\}$ are pairwise disjoint and $\bigcup_{f\in\F}f(V)\subset V$. The set $[0,1]\cup\{2\}$ is an example which has OSC but is not a self regenerating fractal.

\section*{Acknowledgements}

The author would like to thank Taras Banakh for many fruitful discussions and proposing the statement of Definition \ref{def_top_super} and Theorem \ref{main_thm}.


\end{document}